\documentclass[12pt]{amsart} 

\usepackage[english]{babel}
\usepackage[utf8]{inputenc}
\usepackage[T1]{fontenc}

\usepackage{mathtools}
\usepackage{amsmath}
\usepackage{amssymb}
\usepackage{amsfonts}
\usepackage{amssymb}
\usepackage{amsthm}
\usepackage{amscd}
\usepackage{xcolor}
\usepackage{slashed}
\usepackage[hidelinks]{hyperref}
\usepackage{geometry}

\newtheorem{definition}{Definition}[section]
\newtheorem{lemma}[definition]{Lemma}

\newtheorem{theorem}[definition]{Theorem}

\newtheorem{corollary}[definition]{Corollary}
\newtheorem{example}[definition]{Example}
\newtheorem{remark}[definition]{Remark}

\DeclareMathOperator{\Rm}{Rm}
\DeclareMathOperator{\Ric}{Ric}
\DeclareMathOperator{\scal}{R}

\title{A universal Bochner formula for scalar curvature}

\author{Sven Hirsch}
\address{Columbia University, 2990 Broadway, New York NY 10027, USA}
\email{sven.hirsch@columbia.edu }

\begin{document}

\begin{abstract}
We introduce a universal Bochner formula for scalar curvature that contains, as special cases,
the stability inequality for minimal slicings, a Schr\"odinger--Lichnerowicz-type formula,
and a higher-dimensional version of Stern's level-set identity.
\end{abstract}

\maketitle

\section{Introduction}

Many classical results in geometric analysis, such as the positive mass theorem
\cite{Witten1981, SchoenYau79, BrayKazarasKhuriStern2019PMT}
and the Geroch conjecture
\cite{gromov1980spin, SchoenYau79, Stern2022ScalarCurvatureHarmonicMapsS1},
can be proved using three a priori distinct methods: minimal surfaces, spinors, and level sets.
The minimal surface approach relies on stability inequalities, the spinorial method on the
Schr\"odinger--Lichnerowicz formula, and the level-set method on Bochner-type identities.
All three encode scalar curvature, yet they appear conceptually very different.

In this paper our goal is to find a common formulation for all three methods:

\begin{theorem}\label{thm main}
Let $(M^n,g)$ be a Riemannian manifold, and let $u_0,\dots,u_{n-2}$ be $C^3$ functions
satisfying $du_0\wedge\cdots\wedge du_{n-2}\neq 0$.
Define vector fields $Z_0,\dots,Z_{n-1}$ iteratively by
\[
Z_m |Z_{m+1}|^{-1} = \nabla_{\Sigma_m} u_m,
\qquad 0 \le m \le n-2.
\]
Suppose that at a point $p\in M$ we have
\[
\operatorname{div}_{\Sigma_m} Z_m = 0
\qquad \text{for all } 0 \le m \le n-2.
\]
Then at $p$,
\begin{align*}
2|Z_0|^{-1}\Delta|Z_0|
-2\Delta_{\Sigma_{n-1}}\log|Z_{n-1}|
\ge
\scal
+|\nabla\log|Z_0||^2
+2\sum_{m=0}^{n-2}
|Z_m|^{-1}
\nabla^{\Sigma_m}_{\nu_{m+1}}
\bigl(\operatorname{div}_{\Sigma_m}Z_m\bigr).
\end{align*}
\end{theorem}

Here $\Sigma_m$, with unit normals $\nu_1,\dots,\nu_m$, denotes the joint level set of
the map $(u_0,\dots,u_{m-1})$, where $\Sigma_0=M$, and $g_{\Sigma_m}$ is the induced metric.
Although not immediately apparent, Theorem~\ref{thm main} gives, in particular, a new proof
that the torus does not admit a metric of positive scalar curvature.

Below we explain how Theorem~\ref{thm main} recovers the stability inequality for minimal slicings,
a Schr\"odinger--Lichnerowicz-type formula, and a higher-dimensional version of Stern's level-set identity.

Beyond its intrinsic interest, this unified framework has several applications.
It provides a dictionary between the three approaches and their variants; for example,
in the level-set setting one encounters spacetime harmonic functions \cite{HKK},
charged harmonic functions \cite{BHKKZ}, and (spacetime) inverse mean curvature flow
\cite{HuiskenIlmanen97, hirsch2022hawking}.
Results such as Llarull's theorem\footnote{There are several proofs of Llarull's theorem
that do not rely on spinors in low dimensions
\cite{HKKZ, HuLiuShi2023SphericalCaps, CecchiniWangXieZhu2024S4Rigidity};
however, no minimal slicing argument is currently known.} \cite{llarull1998sharp},
the asymptotically AdS positive mass theorem \cite{ChruscielMaertenTod, hirsch2025causal},
the Penrose inequality \cite{HuiskenIlmanen97, Bray01},
and the generalized Geroch conjecture \cite{BrendleHirschJohne24}
have so far been established using only one of these three techniques.
Theorem~\ref{thm main} therefore suggests new avenues for approaching these problems
via alternative methods.
Moreover, it yields genuinely new identities, including a higher-dimensional extension
of Stern's level-set formula, which may lead to further results in scalar curvature geometry.

\subsection{Stable minimal slicings}

Recall from Theorem~\ref{thm main} that the universal Bochner inequality contains the
two Laplacian terms
\[
2|Z_0|^{-1}\Delta|Z_0|
\quad\text{and}\quad
-2\Delta_{\Sigma_{n-1}}\log|Z_{n-1}|.
\]
Since these involve different Laplace operators, only one of them can be controlled
by integration at a time.
Moreover, while there are $n$ vector fields $Z_m$, the system imposes only $n-1$
divergence equations.
This leaves one degree of freedom that can be fixed by normalization.

There are two natural choices:
imposing $|Z_0|=1$ eliminates the first Laplacian term,
while imposing $|Z_{n-1}|=1$ eliminates the second.
The first choice arises naturally from the theory of stable minimal slicings.

\begin{theorem}\label{thm:slicings}
Let $(M^n,g)$ be a smooth Riemannian manifold admitting a stable weighted slicing
\[
\Sigma_{n-1}\subset \cdots \subset \Sigma_0=M
\]
of order $n-1$, cf.~\cite[Definition~1.3]{BrendleHirschJohne24}.
Then there exist smooth functions $u_0,\dots,u_{n-2}$ on $\Sigma_{n-1}$
with $du_0\wedge\cdots\wedge du_{n-2}\neq 0$ and associated vector fields
$Z_m$ defined by
\[
Z_m |Z_{m+1}|^{-1}=\nabla_{\Sigma_m}u_m,
\]
such that on $\Sigma_{n-1}$, for all $0\le m\le n-2$,
\[
\operatorname{div}_{\Sigma_m}Z_m=0,
\qquad
|Z_0|=1,
\qquad
\nabla^{\Sigma_m}_{\nu_{m+1}}\bigl(\operatorname{div}_{\Sigma_m}Z_m\bigr)\ge 0.
\]
\end{theorem}

The condition $\operatorname{div}_{\Sigma_m}Z_m=0$ expresses that $\Sigma_{m+1}$
is a weighted minimal hypersurface in $\Sigma_m$,
while the inequality
\[
\nabla^{\Sigma_m}_{\nu_{m+1}}\bigl(\operatorname{div}_{\Sigma_m}Z_m\bigr)\ge 0
\]
corresponds to the stability of $\Sigma_{m+1}$.

Combining Theorem~\ref{thm:slicings} with Theorem~\ref{thm main}, we immediately
obtain the following nonexistence result, which implies the Geroch conjecture
\cite{gromov1980spin, SchoenYau79}.
In other words, Theorem~\ref{thm main} recovers the stability inequality for
minimal slicings.

\begin{corollary}
Let $(M^n,g)$ be a smooth Riemannian manifold admitting a stable weighted slicing
$\Sigma_{n-1}\subset \cdots \subset \Sigma_0=M$ of order $n-1$,
and let $u_m$ and $Z_m$ be as above.
Then on $\Sigma_{n-1}$,
\[
-2\Delta_{\Sigma_{n-1}}\log|Z_{n-1}|\ge \scal.
\]
In particular, integrating this inequality over $\Sigma_{n-1}$ shows that
$M$ does not admit a metric of positive scalar curvature.
\end{corollary}

\subsection{Spinors and the Lichnerowicz formula}

The study of scalar curvature via spinors hinges on the
Schr\"odinger--Lichnerowicz--Weitzenb\"ock formula
\[
\slashed D^2\psi=\nabla^\ast\nabla\psi+\frac14\scal\,\psi.
\]
If $\psi$ is harmonic, i.e.\ $\slashed D\psi=0$, then Kato's inequality implies
\[
2\Delta|\psi|^2
\ge
\frac{|\nabla|\psi|^2|^2}{|\psi|^2}
+\scal|\psi|^2.
\]
An analogous inequality can be obtained from Theorem~\ref{thm main} by imposing
the normalization $|Z_{n-1}|=1$ (as opposed to $|Z_0|=1$ in the minimal slicing case).

\begin{theorem}\label{thm spinors}
Let $u_0,\dots,u_{n-2}$ be $C^3$ functions satisfying
$du_0\wedge\cdots\wedge du_{n-2}\neq 0$,
and let the associated vector fields $Z_m$ be defined by
\[
Z_m|Z_{m+1}|^{-1}=\nabla_{\Sigma_m}u_m.
\]
Suppose that $|Z_{n-1}|=1$ and that
\[
\operatorname{div}_{\Sigma_m}Z_m=0
\qquad\text{for all }0\le m\le n-2.
\]
Then
\[
2\Delta|Z_0|
\ge
\frac{|\nabla Z_0|^2}{|Z_0|}
+\scal|Z_0|.
\]
\end{theorem}

Thus, one obtains a Lichnerowicz-type inequality in which
$|\psi|^2$ is replaced by $|Z_0|$.
Moreover, if $\psi$ is a harmonic spinor, then the real-valued vector field
\[
X_i=\langle \mathbf i e_i\psi,\psi\rangle
\]
is divergence free, just as $Z_0$ is in Theorem~\ref{thm spinors}.

In dimension $3$ this analogy becomes particularly transparent.
In this case one has $|X|=|\psi|^2$, and hence the
Schr\"odinger--Lichnerowicz--Weitzenb\"ock formula yields
\[
2\Delta|X|
\ge
\frac{|\nabla X|^2}{|X|}
+\scal|X|
\]
for the divergence-free vector field $X$.

Furthermore, exploiting the quaternionic structure of the spinor bundle in
dimension $3$, a single harmonic spinor gives rise to three divergence-free
vector fields, cf.\ Lemma~\ref{Spinors dimension 3}.
These vector fields are pairwise orthogonal and have equal length.
The same phenomenon occurs in Theorem~\ref{thm main}:

\begin{theorem}\label{thm divergence free vector fields}
Let $Z_m$ be as in Theorem~\ref{thm main}, and define
\[
Y_m = Z_m |Z_m|^{-1}|Z_0|.
\]
Then
\[
\operatorname{div} Y_m = 0.
\]
Moreover, the vector fields $Y_m$ have the same length and are mutually orthogonal.
\end{theorem}

We emphasize that Theorem~\ref{thm divergence free vector fields} also applies in
the setting of minimal slicings.
In this case the vector fields $Y_m$ coincide with the unit normals
$\nu_{m+1}$ of the hypersurfaces $\Sigma_{m+1}$.
Thus, much like the Dirac equation, minimal slicings can be interpreted as a
mechanism for constructing pairwise orthogonal divergence-free vector fields
of equal length.

We also derive a more general version of
Theorem~\ref{thm main} formulated as an equality similar to the Schr\"odinger--Lichnerowicz--Weitzenb\"ock formula, cf. Theorem~\ref{thm main more general}.

\subsection{Level sets and intermediate curvature}

So far we have focused on the case of $(n-1)$ functions
$u_0,\dots,u_{n-2}$ with associated vector fields $Z_m$.
However, Theorem~\ref{thm main} admits a natural generalization to an
arbitrary number of functions $u_0,\dots,u_{s-1}$;
see Theorem~\ref{thm intermediate}.

For $s=1$ together with $|Z_{s}|=1$ this generalization reduces to the classical Bochner formula
for Ricci curvature.
At the other extreme, when $s=n-2$ and $|Z_s|=1$, it yields a higher-dimensional
generalization of Stern's level-set identity:

\begin{theorem}\label{thm stern}
Let $u_0,\dots,u_{n-3}$ be $C^3$ functions on $M$ satisfying
$du_0\wedge\cdots\wedge du_{n-3}\neq 0$, and let the associated vector
fields $Z_m$ be defined by
\[
Z_m|Z_{m+1}|^{-1}=\nabla_{\Sigma_m}u_m.
\]
Assume that $|Z_{s-2}|=1$ and for all $0\le m\le n-3$
\[
\operatorname{div}_{\Sigma_m}Z_m=0.
\]
Then, denoting by $\operatorname{K}(\Sigma_{n-2})$ the Gaussian curvature of $\Sigma_{n-2}$,
we have
\[
2\Delta|Z_0|
\ge
\bigl(\scal_{\Sigma_0}-2\operatorname{K}(\Sigma_{n-2})\bigr)|Z_0|.
\]
In particular, if the map $\mathbf u=(u_0,\dots,u_{n-3})$ is $T^{n-2}$-valued, then
integration, together with the coarea formula and the Gauss--Bonnet theorem,
yields
\[
4\pi\int_{\mathbf t\in T^{n-2}}
\chi\bigl(\{x:\mathbf u(x)=\mathbf t\}\bigr)
\ge
\int_{M^n}\scal\,|du_0\wedge\cdots\wedge du_{n-3}|,
\]
where $\chi$ denotes the Euler
characteristic of the level set $\{x:\mathbf u(x)=\mathbf t\}$.
\end{theorem}

For $n=3$ this reduces to Stern's identity \cite{Stern2022ScalarCurvatureHarmonicMapsS1},
\[
4\pi\int_{t\in S^1}\chi(\{x:u(x)=t\})
\ge
\int_{T^3}\scal \,|du|.
\]
Once solvability of the associated PDE system has been established in higher
dimensions, this framework may provide a new approach to scalar curvature
problems in all dimensions, including the positive mass theorem and the
Geroch conjecture.
Also see Remark \ref{rem pde systems} for an alternative formulation of the PDE system $\operatorname{div}_{\Sigma_m}Z_m=0$.

We emphasize that the nonvanishing comass condition
$du_0\wedge\cdots\wedge du_{n-3}\neq 0$ is analogous to the assumption
of a nonvanishing gradient in the classical level-set method.
In those settings, the geometric conclusions remain valid even when
the gradient vanishes, although the analysis becomes more delicate;
see for instance \cite{Stern2022ScalarCurvatureHarmonicMapsS1, HKK, HuiskenIlmanen97}.

Finally, for other values of $s$, Theorem~\ref{thm intermediate}
encodes $s$-intermediate curvature in the sense of
\cite{BrendleHirschJohne24}.
Thus, Theorem~\ref{thm intermediate} provides a new analytic approach
to the study of intermediate curvature, complementary to the method
of minimal slicings.

\textbf{Acknowledgements:} This work was supported in part by Grant No. DMS-1926686 of the National Science Foundation, and the School of Mathematics of the Institute of Advanced Study.
The author is grateful to Bernd Ammann, Misha Gromov and Bernhard Hanke for helpful discussions and their interest in this work.

\section{Proof of the Bochner formula}

Throughout this section we work in the setting of Theorem~\ref{thm main}.
Let $(M^n,g)$ be a smooth Riemannian manifold and fix a point $p\in M$.
Let $u_0,\dots,u_{n-2}$ be $C^3$ functions on $M$ such that
\[
du_0\wedge\cdots\wedge du_{n-2}\neq 0
\quad\text{at }p.
\]
For $m\ge 0$ let $\Sigma_m$ denote the joint level set of the map
$(u_0,\dots,u_{m-1})$, with the convention $\Sigma_0=M$.
Write $g_{\Sigma_m}$ for the induced metric on $\Sigma_m$ and
$\nabla^{\Sigma_m}$ for its Levi--Civita connection.
For $0\le m\le n-2$, let $\nu_{m+1}$ be the unit normal of
$\Sigma_{m+1}\subset\Sigma_m$.

Let $Z_0,\dots,Z_{n-1}$ be nonvanishing smooth vector fields defined iteratively by
\[
Z_m|Z_{m+1}|^{-1}=\nabla_{\Sigma_m}u_m,
\qquad 0\le m\le n-2,
\]
and assume that at $p$ we have
\[
\operatorname{div}_{\Sigma_m}Z_m=0
\qquad\text{for all }0\le m\le n-2.
\]

\begin{lemma}\label{Lemma: H}
At $p$ we have
\[
H_{\Sigma_{m+1}}
=
-\nabla^{\Sigma_m}_{\nu_{m+1}}\log|Z_m|
\]
and
\[
|A_{\Sigma_{m+1}}|^2
=
|\nabla_{\Sigma_m}u_m|^{-2}|\nabla^2_{\Sigma_m}u_m|^2
-2|\nabla_{\Sigma_m}\log|\nabla_{\Sigma_m}u_m||^2
+\bigl(\nabla_{\nu_{m+1}}^{\Sigma_m}\log|\nabla_{\Sigma_m}u_m|\bigr)^2.
\]
\end{lemma}

\begin{proof}
Since, $     \operatorname{div}_{\Sigma_m}Z_m=0$ we obtain 
\begin{align*}
-|Z_{m+1}|^{-2}\langle Z_m, \nabla_{\Sigma_m}|Z_{m+1}| \rangle=  \operatorname{div}_{\Sigma_m}(Z_m|Z_{m+1}|^{-1})=\Delta_{\Sigma_m}u_m.
\end{align*}
Consequently,
\begin{align*}
    -\nabla^{\Sigma_m}_{\nu_{m+1}}\log|Z_m|
    =&-|Z_m|^{-2}\langle \nabla^{\Sigma_m}_{\nu_{m+1}}Z_m,Z_m\rangle\\
    =&-|Z_m|^{-2}\langle \nabla^{\Sigma_m}_{\nu_{m+1}}(|Z_{m+1}|\nabla_{\Sigma_m}u_m),Z_m\rangle\\
    =&|Z_m|^{-1}|Z_{m+1}|(\Delta_{\Sigma_m}-\nabla^{\Sigma_m}_{\nu_{m+1}\nu_{m+1}})u_m\\
    =& H_{\Sigma_{m+1}}
\end{align*}
where we used $Z_{m}|Z_{m+1}|^{-1}=\nabla_{\Sigma_m}u_m$.
The second identity follows from the well-known equation $A_{\Sigma_{m+1}}=|\nabla^{\Sigma_m}u_m|^{-1}\nabla_{\Sigma_m}^2u_m$.
\end{proof}

\begin{lemma}\label{Lemma: A-H identity}
We have at $p$
    \begin{align*}
        &|A_{\Sigma_{m+1}}|^2-H_{\Sigma_{m+1}}^2           \\
        =&|\nabla_{\Sigma_m}\nu_m|^2+|\nabla_{\Sigma_{m+1}}\log|Z_{m+1}||^2
        +2\langle\nabla_{\Sigma_{m+1}}|Z_{m+1}|,\nabla_{\Sigma_{m+1}}\log|\nabla_{\Sigma_m}u_m|\rangle
        -|\nabla_{\Sigma_m}\log|Z_m||^2
    \end{align*}
\end{lemma}

\begin{proof}
Using $Z_{m}|Z_{m+1}|^{-1}=\nabla_{\Sigma_m}u_m$, we compute at $p$
\begin{align*}
   &|\nabla_{\Sigma_m}u_m|^{-2}|\nabla^2_{\Sigma_m}u_m|^2\\
   =&|\nabla_{\Sigma_m}Z_m|^2|Z_m|^{-2}+|\nabla_{\Sigma_m}\log|Z_{m+1}||^2-2\langle\nabla_{\Sigma_m}\log|Z_{m+1}|,\nabla_{\Sigma_m}\log |Z_m|\rangle
\end{align*}
and
\begin{align*}
    &-2|\nabla_{\Sigma_m}\log|\nabla_{\Sigma_m}u_m||^2\\
    =&-2|\nabla_{\Sigma_m}\log|Z_m||^2+4\langle \nabla_{\Sigma_m}\log|Z_m|,\nabla_{\Sigma_m}\log|Z_{m+1}|\rangle-2|\nabla_{\Sigma_m}\log|Z_{m+1}||^2
\end{align*}
Combining this with $|\nabla_{\Sigma_m}Z_m|^2|Z_m|^{-2}=|\nabla_{\Sigma_m}\nu_m|^2+|\nabla_{\Sigma_m}\log |Z_m||^2$ and Lemma \ref{Lemma: H} yields
        \begin{align*}
        &|A_{\Sigma_{m+1}}|^2\\
              =&|\nabla_{\Sigma_m}\nu_m|^2-|\nabla_{\Sigma_m}\log|Z_{m+1}||^2-|\nabla_{\Sigma_m}\log|Z_m||^2\\
        &+2\langle \nabla_{\Sigma_m}\log|Z_m|,\nabla_{\Sigma_m}\log|Z_{m+1}|\rangle+(\nabla_{\nu_{m+1}}^{\Sigma_m}\log|\nabla_{\Sigma_m}u_m|)^2\\
            =&|\nabla_{\Sigma_m}\nu_m|^2+|\nabla_{\Sigma_m}\log|Z_{m+1}||^2-|\nabla_{\Sigma_m}\log|Z_m||^2\\&+2\langle \nabla_{\Sigma_m}\log|\nabla_{\Sigma_m}u_m|,\nabla_{\Sigma_m}\log|Z_{m+1}|\rangle+(\nabla_{\nu_{m+1}}^{\Sigma_m}\log|\nabla_{\Sigma_m}u_m|)^2
     \end{align*}
     where we used again $Z_{m}|Z_{m+1}|^{-1}=\nabla_{\Sigma_m}u_m$ in the last equation.
    Moreover,
    \begin{align*}
        H_{\Sigma_{m+1}}^2=&|\nabla^{\Sigma_m}_{\nu_{m+1}}\log|Z_m||^2\\
        =&|\nabla^{\Sigma_m}_{\nu_{m+1}}\log|Z_{m+1}||^2+|\nabla^{\Sigma_m}_{\nu_{m+1}}\log|\nabla_{\Sigma_m}u_m||^2+2\nabla^{\Sigma_m}_{\nu_{m+1}}\log|Z_{m+1}|\nabla^{\Sigma_m}_{\nu_{m+1}}\log|\nabla_{\Sigma_m}u_m|
    \end{align*}
    Combining terms, the result follows.
\end{proof}

\begin{lemma}\label{Lemma:main formula}
  We have at $p$
\begin{align*}
      &\Delta_{\Sigma_m}\log|Z_m|-\Delta_{\Sigma_{m+1}}\log|Z_{m+1}|\\
        =&\Ric_{\Sigma_m}(\nu_{m+1},\nu_{m+1})+|A_{\Sigma_{m+1}}|^2-H_{\Sigma_{m+1}}^2\\
        &+|Z_m|^{-1}\nabla^{\Sigma_m}_{\nu_{m+1}}(\operatorname{div}_{\Sigma_m}Z_m)-\langle \nabla_{\Sigma_{m+1}}\log|\nabla_{\Sigma_m}u_m|,\nabla_{\Sigma_{m+1}}\log|Z_{m+1}|\rangle
        \end{align*}
\end{lemma}

\begin{proof}
Using $Z_{m}|Z_{m+1}|^{-1}=\nabla_{\Sigma_m}u_m$, we obtain
\begin{align*}
      &\Delta_{\Sigma_m}\log|Z_m|-\Delta_{\Sigma_{m+1}}\log|Z_{m+1}|\\
        =&\Delta_{\Sigma_m}\log|Z_{m+1}|+\Delta_{\Sigma_m}\log|\nabla_{\Sigma_m}u_m|-(\Delta_{\Sigma_m}-\nabla^{\Sigma_m}_{\nu_{m+1}\nu_{m+1}}-H_{\Sigma_{m+1}}\nabla^{\Sigma_m}_{\nu_{m+1}})\log|Z_{m+1}|\\
            =&\Delta_{\Sigma_m}\log|\nabla_{\Sigma_m}u_m|+(\nabla^{\Sigma_m}_{\nu_{m+1}\nu_{m+1}}+H_{\Sigma_{m+1}}\nabla^{\Sigma_m}_{\nu_{m+1}})\log|Z_{m+1}|.
\end{align*}
Bochner's formula yields
\begin{align*}
    \Delta_{\Sigma_m}\log|\nabla_{\Sigma_m}u_m|=&\Ric_{\Sigma_m}(\nu_{m+1},\nu_{m+1})+|\nabla^2_{\Sigma_m}u_m|^2|\nabla_{\Sigma_m}u_m|^{-2}-2|\nabla_{\Sigma_m}\log|\nabla_{\Sigma_m}u_m||^2\\
        &+|\nabla_{\Sigma_m}u_m|^{-1}\nabla^{\Sigma_m}_{\nu_{m+1}}\Delta_{\Sigma_m}u_m.
\end{align*}
Moreover,
\begin{align*}
    &|\nabla_{\Sigma_m}u_m|^{-1}\nabla^{\Sigma_m}_{\nu_{m+1}}\Delta_{\Sigma_m}u_m\\=&
     |\nabla_{\Sigma_m}u_m|^{-1}\nabla^{\Sigma_m}_{\nu_{m+1}}(\operatorname{div}_{\Sigma_m}(Z_m|Z_{m+1}|^{-1}))\\
     =&  |Z_m|^{-1}\nabla^{\Sigma_m}_{\nu_{m+1}}(\operatorname{div}_{\Sigma_m}Z_m)+|\nabla_{\Sigma_m}u_m|^{-1}\nabla^{\Sigma_m}_{\nu_{m+1}}(|Z_m|\nabla^{\Sigma_m}_{\nu_{m+1}}|Z_{m+1}|^{-1})\\
     =&|Z_m|^{-1}\nabla^{\Sigma_m}_{\nu_{m+1}}(\operatorname{div}_{\Sigma_m}Z_m)-
    |\nabla_{\Sigma_m}u_m|^{-1} \nabla_{\nu_{m+1}}^{\Sigma_m}(|\nabla_{\Sigma_m}u_m|\nabla^{\Sigma_m}_{\nu_{m+1}}\log|Z_{m+1}|)
\end{align*}
Furthermore, recall from Lemma \ref{Lemma: H}
\begin{align*}
    |A_{\Sigma_{m+1}}|^2
        =&|\nabla_{\Sigma_m}u_m|^{-2}|\nabla^2_{\Sigma_m}u_m|^2-2|\nabla_{\Sigma_m}\log|\nabla_{\Sigma_m}u_m||^2+(\nabla_{\nu_{m+1}}^{\Sigma_m}\log|\nabla_{\Sigma_m}u_m|)^2
\end{align*}
Consequently,
\begin{align*}
      \Delta_{\Sigma_m}\log|\nabla_{\Sigma_m}u_m|
      =&\Ric_{\Sigma_m}(\nu_{m+1},\nu_{m+1})+ |A_{\Sigma_{m+1}}|^2-(\nabla_{\nu_{m+1}}^{\Sigma_m}\log|\nabla_{\Sigma_m}u_m|)^2\\
        &+|Z_m|^{-1}\nabla^{\Sigma_m}_{\nu_{m+1}}(\operatorname{div}_{\Sigma_m}Z_m)-
    |\nabla_{\Sigma_m}u_m|^{-1} \nabla_{\nu_{m+1}}^{\Sigma_m}(|\nabla_{\Sigma_m}u_m|\nabla^{\Sigma_m}_{\nu_{m+1}}\log|Z_{m+1}|)
\end{align*}
We also have
\begin{align*}
   \nabla^{\Sigma_m}_{\nu_{m+1}\nu_{m+1}}\log|Z_{m+1}|=\nabla_{\nu_{m+1}}^{\Sigma_m}(\nabla^{\Sigma_m}_{\nu_{m+1}}\log|Z_{m+1}|)-\langle \nabla_{\Sigma_{m+1}}\log|\nabla_{\Sigma_m}u_m|,\nabla_{\Sigma_{m+1}}\log|Z_{m+1}|\rangle.
\end{align*}
Moreover, by Lemma \ref{Lemma: H} and $Z_{m}|Z_{m+1}|^{-1}=\nabla_{\Sigma_m}u_m$ 
\begin{align*}
    H_{\Sigma_{m+1}}\nabla^{\Sigma_m}_{\nu_{m+1}}\log|Z_{m+1}|=&-H_{\Sigma_{m+1}}^2+\nabla^{\Sigma_m}_{\nu_{m+1}}\log|Z_{m}|\nabla^{\Sigma_m}_{\nu_{m+1}}\log|\nabla^{\Sigma_m}u_m|
\end{align*}
Combining everything yields
\begin{align*}
      &\Delta_{\Sigma_m}\log|Z_m|-\Delta_{\Sigma_{m+1}}\log|Z_{m+1}|\\
        =&\Ric_{\Sigma_m}(\nu_{m+1},\nu_{m+1})+|A_{\Sigma_{m+1}}|^2-(\nabla_{\nu_{m+1}}^{\Sigma_m}\log|\nabla_{\Sigma_m}u_m|)^2\\
        &+|Z_m|^{-1}\nabla^{\Sigma_m}_{\nu_{m+1}}(\operatorname{div}_{\Sigma_m}Z_m)-
    |\nabla_{\Sigma_m}u_m|^{-1} \nabla_{\nu_{m+1}}^{\Sigma_m}(|\nabla_{\Sigma_m}u_m|\nabla^{\Sigma_m}_{\nu_{m+1}}\log|Z_{m+1}|)
        \\&+\nabla_{\nu_{m+1}}^{\Sigma_m}(\nabla^{\Sigma_m}_{\nu_{m+1}}\log|Z_{m+1}|)-\langle \nabla_{\Sigma_{m+1}}\log|\nabla_{\Sigma_m}u_m|,\nabla_{\Sigma_{m+1}}\log|Z_{m+1}|\rangle
        \\&-H_{\Sigma_{m+1}}^2+\nabla^{\Sigma_m}_{\nu_{m+1}}\log|Z_{m}|\nabla^{\Sigma_m}_{\nu_{m+1}}\log|\nabla^{\Sigma_m}u_m|.
        \end{align*}
Exploiting that the two $\nabla_{\nu_{m+1}}^{\Sigma_m}(\nabla^{\Sigma_m}_{\nu_{m+1}}\log|Z_{m+1}|)$ terms cancel and using once more that $|Z_m|=|Z_{m+1}||\nabla_{\Sigma_m}u_m|$, the result follows.
\end{proof}

\begin{corollary}\label{cor:main}
    We have at $p$
    \begin{align*}
      &\Delta_{\Sigma_m}\log|Z_m|-\Delta_{\Sigma_{m+1}}\log|Z_{m+1}|\\
        =&\Ric_{\Sigma_m}(\nu_{m+1},\nu_{m+1})+\frac12|A_{\Sigma_{m+1}}|^2-\frac12H_{\Sigma_{m+1}}^2\\
        &+|Z_m|^{-1}\nabla^{\Sigma_m}_{\nu_{m+1}}(\operatorname{div}_{\Sigma_m}Z_m)+\frac12|\nabla_{\Sigma_m}\nu_m|^2+\frac12|\nabla_{\Sigma_{m+1}}\log|Z_{m+1}||^2
               -\frac12|\nabla_{\Sigma_m}\log|Z_m||^2
        \end{align*}
\end{corollary}

\begin{proof}
    This immediately follows from combining Lemma \ref{Lemma: H} with Lemma \ref{Lemma:main formula}.
\end{proof}

\begin{theorem}\label{thm main more general}
    We have at $p$
    \begin{align*}
    & 2 |Z_0|^{-1}\Delta  |Z_0|-2\Delta_{\Sigma_{n-1}}\log |Z_{n-1}-2\sum_{m=0}^{n-2}|Z_m|^{-1}\nabla^{\Sigma_m}_{\nu_{m+1}}(\operatorname{div}_{\Sigma_m}Z_m)\\
     =&\scal+\sum_{m=0}^{n-2}|\nabla_{\Sigma_m}\nu_m|^2+|\nabla_{\Sigma_{n-1}}\log|Z_{n-1}||^2+|\nabla_{\Sigma_0}\log|Z_0||^2
\end{align*}
\end{theorem}

\begin{proof}
    Summing the formula from Corollary \ref{cor:main} over $m$ yields
\begin{align*}
    & 2 \Delta \log |Z_0|-2\Delta_{\Sigma_{n-1}}\log |Z_{n-1}-2\sum_{m=0}^{n-2}|Z_m|^{-1}\nabla^{\Sigma_m}_{\nu_{m+1}}(\operatorname{div}_{\Sigma_m}Z_m)\\
     =&\scal+\sum_{m=0}^{n-2}|\nabla_{\Sigma_m}\nu_m|^2+|\nabla_{\Sigma_{n-1}}\log|Z_{n-1}||^2-|\nabla \log|Z_0||^2
\end{align*}
    where we used that 
    \begin{align*}
       \sum_{m=0}^{n-2}\left( \Ric_{\Sigma_m}(\nu_{m+1},\nu_{m+1})+\frac12|A_{\Sigma_{m+1}}|^2-\frac12H_{\Sigma_{m+1}}^2\right)=\frac12\scal,
    \end{align*}
    cf. \cite[Remark 3.9]{BrendleHirschJohne24}.
    Using that $\Delta\log|Z_0|=|Z_0|^{-1}\Delta |Z_0|-|\nabla \log|Z_0||^2$, the result follows.
\end{proof}

\begin{proof}[Proof of Theorem \ref{thm main}]
    This follows immediately from Theorem \ref{thm main more general}.
\end{proof}

\section{Minimal slicings, spinors, and level-sets}

\begin{proof}[Proof of Theorem~\ref{thm:slicings}]
Let
\[
\Sigma_{n-1}\subset \Sigma_{n-2}\subset \cdots \subset \Sigma_1\subset \Sigma_0=M
\]
be a stable weighted minimal slicing with associated weight functions $\rho_m$.
By definition, $\rho_0=1$, and for each $1\le m\le n-1$,
the hypersurface $\Sigma_m$ is an embedded two-sided hypersurface in $\Sigma_{m-1}$
and a stable critical point of the $\rho_{m-1}$-weighted area functional
\[
\mathcal H^{n-m}_{\rho_{m-1}}(\Sigma)
=
\int_\Sigma \rho_{m-1}\, d\mu
\]
among hypersurfaces $\Sigma\subset\Sigma_{m-1}$.
Moreover, for each $1\le m\le n-1$, the function
\[
v_m=\frac{\rho_m}{\rho_{m-1}|_{\Sigma_m}}\in C^\infty(\Sigma_m)
\]
is a first eigenfunction of the stability operator associated with the
$\rho_{m-1}$-weighted area functional.

Choose functions $u_m$ on $M$ such that
\[
u_m|_{\Sigma_{m+1}}=0,
\qquad
\nabla_{\Sigma_m}u_m|_{\Sigma_{m+1}}=v_{m+1}^{-1}\nu_{m+1}.
\]
Define vector fields $Z_m$ by
\[
Z_m=|Z_{m+1}|\nabla_{\Sigma_m}u_m
\qquad\text{for }0\le m\le n-2,
\]
and set $|Z_{n-1}|=\rho_{n-1}$.

We claim that
\[
|Z_m|=\rho_m
\qquad\text{for all }0\le m\le n-1.
\]
By definition this holds for $m=n-1$.
Assuming it holds for $m+1$, we compute
\[
|Z_m|
=
|Z_{m+1}||\nabla_{\Sigma_m}u_m|
=
\rho_{m+1}v_{m+1}
=
\rho_m.
\]
In particular, $|Z_0|=1$.

For a weighted minimal slicing it is well known that
\[
H_{\Sigma_{m+1}}
=
-\nabla_{\nu_{m+1}}\log\rho_m.
\]
Arguing as in the proof of Lemma~\ref{Lemma: H}, we compute
\begin{align*}
-|Z_{m+1}|^{-2}\langle Z_m,\nabla_{\Sigma_m}|Z_{m+1}|\rangle
&=
\operatorname{div}_{\Sigma_m}(Z_m|Z_{m+1}|^{-1})
-
|Z_{m+1}|^{-1}\operatorname{div}_{\Sigma_m}Z_m \\
&=
\Delta_{\Sigma_m}u_m
-
|Z_{m+1}|^{-1}\operatorname{div}_{\Sigma_m}Z_m.
\end{align*}
Consequently,
\begin{align*}
-\nabla^{\Sigma_m}_{\nu_{m+1}}\log|Z_m|
&=
-|Z_m|^{-2}\langle\nabla^{\Sigma_m}_{\nu_{m+1}}Z_m,Z_m\rangle \\
&=
-|Z_m|^{-2}
\Big\langle
\nabla^{\Sigma_m}_{\nu_{m+1}}\bigl(|Z_{m+1}|\nabla_{\Sigma_m}u_m\bigr),
Z_m
\Big\rangle \\
&=
|Z_m|^{-1}|Z_{m+1}|
(\Delta_{\Sigma_m}-\nabla^{\Sigma_m}_{\nu_{m+1}\nu_{m+1}})u_m
-
|Z_m|^{-1}\operatorname{div}_{\Sigma_m}Z_m \\
&=
H_{\Sigma_{m+1}}
-
|Z_m|^{-1}\operatorname{div}_{\Sigma_m}Z_m,
\end{align*}
where we used $Z_m|Z_{m+1}|^{-1}=\nabla_{\Sigma_m}u_m$.
Since $H_{\Sigma_{m+1}}=-\nabla_{\nu_{m+1}}\log\rho_m$ and $\rho_m=|Z_m|$,
it follows that
\[
\operatorname{div}_{\Sigma_m}Z_m=0.
\]

Next, we compute
\begin{align*}
    &\nabla_{\nu_{m+1}} (\operatorname{div}_{\Sigma_m}Z_m)\\
    =&\nabla_{\nu_{m+1}}(\operatorname{div}_{\Sigma_m}(\nu_{m+1}\rho_m)\\
    =&\nabla_{\nu_{m+1}}(H_{\Sigma_{m+1}}\rho_m+\nabla_{\nu_{m+1}}\rho_m)\\
    =&- \Delta_{\Sigma_{m+1}} v_{m+1} 
   - \langle \nabla_{\Sigma_{m+1}} \log \rho_{m}, \nabla_{\Sigma_{m+1}} v_{m+1} \rangle\\
  &- \left( 
  |A_{\Sigma_{m+1}}|^2 + \Ric_{\Sigma_{m}}(\nu_{m+1}, \nu_{m+1}) \right) v_{m+1} \\
  &+ (\nabla_{\Sigma_{m}}^2 \log \rho_{m})(\nu_{m+1}, \nu_{m+1})
   v_m.
\end{align*}
The Jacobi equation for $v_{m+1}$ implies that the above term equals to
   \begin{align*}
   \lambda_{m+1} v_{m+1}.
 \end{align*}
Since $\lambda_{m+1}\ge0$ by stability, the result follows.
\end{proof}

\begin{proof}[Proof of Theorem~\ref{thm spinors}]
This follows directly from Theorem~\ref{thm main} by imposing the normalization
$|Z_{n-1}|=1$, which eliminates the boundary Laplacian term and yields the
Lichnerowicz-type inequality stated in Theorem~\ref{thm spinors}.
\end{proof}

\begin{lemma}\label{lemma divergence}
Let $W$ be a vector field on $\Sigma_m$ that is tangential to $\Sigma_{m+1}$.
If
\[
\operatorname{div}_{\Sigma_{m+1}}W=0,
\]
then
\[
\operatorname{div}_{\Sigma_m}\bigl(|\nabla_{\Sigma_m}u_m|\,W\bigr)=0.
\]
\end{lemma}

\begin{proof}
We first compute the divergence of $W$ in $\Sigma_m$.
Using the standard decomposition of the divergence under a hypersurface
inclusion $\Sigma_{m+1}\subset\Sigma_m$, we have
\begin{align*}
\operatorname{div}_{\Sigma_m}W
&=
\operatorname{div}_{\Sigma_{m+1}}W
+H_{\Sigma_{m+1}}\langle W,\nu_{m+1}\rangle
+\nabla^{\Sigma_m}_{\nu_{m+1}}\langle W,\nu_{m+1}\rangle.
\end{align*}
The first term vanishes by assumption, and the second term vanishes since
$W$ is tangential to $\Sigma_{m+1}$.
Hence,
\[
\operatorname{div}_{\Sigma_m}W
=
\nabla^{\Sigma_m}_{\nu_{m+1}}\langle W,\nu_{m+1}\rangle.
\]

Next, recall that
\[
\nu_{m+1}
=
\frac{\nabla_{\Sigma_m}u_m}{|\nabla_{\Sigma_m}u_m|}.
\]
Therefore,
\begin{align*}
\nabla^{\Sigma_m}_{\nu_{m+1}}\langle W,\nu_{m+1}\rangle
&=
\nabla^{\Sigma_m}_{\nu_{m+1}}
\Big\langle
W,
\nabla_{\Sigma_m}u_m\,|\nabla_{\Sigma_m}u_m|^{-1}
\Big\rangle \\
&=
-\Big\langle
W,
\nabla^{\Sigma_m}_{\nu_{m+1}}
\bigl(\nabla_{\Sigma_m}u_m\,|\nabla_{\Sigma_m}u_m|^{-1}\bigr)
\Big\rangle,
\end{align*}
where we used that $W$ is tangential to $\Sigma_{m+1}$ and hence orthogonal
to $\nabla_{\Sigma_m}u_m$ along $\Sigma_{m+1}$.

Consequently,
\[
\operatorname{div}_{\Sigma_m}W
=
-|\nabla_{\Sigma_m}u_m|^{-1}
\bigl\langle
W,
\nabla^{\Sigma_m}_{\nu_{m+1}}\nabla_{\Sigma_m}u_m
\bigr\rangle.
\]

Finally, we compute
\begin{align*}
\operatorname{div}_{\Sigma_m}\bigl(|\nabla_{\Sigma_m}u_m|\,W\bigr)
&=
|\nabla_{\Sigma_m}u_m|\operatorname{div}_{\Sigma_m}W
+\bigl\langle\nabla_{\Sigma_m}|\nabla_{\Sigma_m}u_m|,W\bigr\rangle \\
&=
- \bigl\langle
W,
\nabla^{\Sigma_m}_{\nu_{m+1}}\nabla_{\Sigma_m}u_m
\bigr\rangle
+
\bigl\langle\nabla_{\Sigma_m}|\nabla_{\Sigma_m}u_m|,W\bigr\rangle
=0.
\end{align*}
\end{proof}

\begin{proof}[Proof of Theorem~\ref{thm divergence free vector fields}]
By construction, the vector fields $Y_m$ are mutually orthogonal and have
the same length.
The divergence-free property
\[
\operatorname{div}Y_m=0
\]
follows directly from Lemma~\ref{lemma divergence}.
\end{proof}

\begin{lemma}\label{Dirac current harmonic}
Let $(M,g)$ be a spin manifold, and let $\psi\in\mathcal S(M)$ be a harmonic spinor.
Define the Dirac current $X$ by
\[
X_i=\langle \mathbf{i}\, e_i\psi,\psi\rangle.
\]
Then $X$ is real-valued and satisfies
\[
\operatorname{div}X=0.
\]
\end{lemma}

\begin{proof}
We use that the Hermitian inner product on the spinor bundle is linear in the
second argument and conjugate linear in the first, and that Clifford
multiplication by $e_i$ is skew-adjoint.
Thus,
\[
\langle \mathbf{i}\, e_i\psi,\psi\rangle
=
\langle \psi,\mathbf{i}\, e_i\psi\rangle,
\]
which implies that $X_i$ is real-valued.

Next, we compute the divergence of $X$:
\begin{align*}
\operatorname{div}X
&=
\nabla_i\bigl(\langle \mathbf{i}\, e_i\psi,\psi\rangle\bigr) \\
&=
\langle \mathbf{i}\, e_i\nabla_i\psi,\psi\rangle
+
\langle \mathbf{i}\, e_i\psi,\nabla_i\psi\rangle \\
&=
2\,\operatorname{Re}\langle \mathbf{i}\, e_i\nabla_i\psi,\psi\rangle
=
2\,\operatorname{Re}\langle \mathbf{i}\,\slashed D\psi,\psi\rangle.
\end{align*}
Since $\psi$ is harmonic, $\slashed D\psi=0$, and therefore $\operatorname{div}X=0$.
\end{proof}

\begin{lemma}\label{Spinors dimension 3}
Let $(M,g)$ be a $3$-dimensional spin manifold with spinor bundle $\mathcal S(M)$.
Let $\psi$ be a harmonic spinor, and let $X$ be its Dirac current.
Then there exist two additional vector fields $A$ and $B$ such that
\[
\operatorname{div}(A)=0,
\qquad
\operatorname{div}(B)=0,
\]
and such that $X$, $A$, and $B$ are pairwise orthogonal with
\[
|\psi|^2=|X|=|A|=|B|.
\]
\end{lemma}

We point out that, in two dimensions, the Dirac equation becomes the Cauchy-Riemann equations $\partial_xu=\partial_yv$, $\partial_yu=-\partial_xv$ which also give rise to two divergence free vector fields $\nabla u,\nabla v$ which again are perpendicular and have the same length.
There is also a similar construction in 3+1 dimensions which yields four such vector fields \cite{tod1983all}.

\begin{proof}
Since $\dim M=3$, the spinor bundle $\mathcal S(M)$ carries a natural quaternionic
structure. In particular, there exists an antilinear bundle endomorphism
\[
\mathcal J:\mathcal S(M)\to\mathcal S(M)
\]
which commutes with Clifford multiplication and with the spin connection $\nabla$, cf. \cite[Remark 2.13]{Ammann2007MassEndomorphism}.
Using this structure, we define two additional vector fields by
\[
A_i=\operatorname{Re}\langle e_i\mathcal  J\psi,\psi\rangle,
\qquad
B_i=\operatorname{Im}\langle e_i \mathcal J\psi,\psi\rangle.
\]

Locally, near any point $p\in M$, we may identify the spinor $\psi$ with a pair of
complex numbers and write
\[
\psi=
\begin{pmatrix}
a+b\mathbf i\\
c+d\mathbf i
\end{pmatrix}.
\]
Recall that the Pauli matrices are given by
\[
\sigma_1=
\begin{pmatrix}
0 & 1\\
1 & 0
\end{pmatrix},
\quad
\sigma_2=
\begin{pmatrix}
0 & -\mathbf i\\
\mathbf i & 0
\end{pmatrix},
\quad
\sigma_3=
\begin{pmatrix}
1 & 0\\
0 & -1
\end{pmatrix},
\]
and that Clifford multiplication is represented by
\[
e_i=\mathbf i\,\sigma_i.
\]
Moreover, the quaternionic structure $J$ can be written explicitly as
\[
\mathcal J\phi=e_2\,\overline{\phi},
\]
where $\overline{\phi}$ denotes complex conjugation.
With these conventions, the Dirac equation
\[
\slashed D\psi=\sum_{i=1}^3 e_i\nabla_i\psi=0
\]
is equivalent to the system of divergence-free conditions
\begin{align*}
\operatorname{div}\begin{pmatrix}-d\\ c \\ -b\end{pmatrix}=
\operatorname{div}\begin{pmatrix}c\\d\\a\end{pmatrix}=
\operatorname{div}\begin{pmatrix}-b\\-a \\d\end{pmatrix}=
\operatorname{div}\begin{pmatrix}a\\ -b\\ -c\end{pmatrix}=0.
\end{align*}

We first compute the components of the Dirac current
$X_i=\langle \mathbf i e_i\psi,\psi\rangle$:

\begin{align*}
X_1
=&\left\langle \mathbf i e_1\psi,\psi\right\rangle
=\left\langle
-\begin{pmatrix}0&1\\ 1&0\end{pmatrix}
\begin{pmatrix}a+b\mathbf i\\ c+d\mathbf i\end{pmatrix},
\begin{pmatrix}a-b\mathbf i\\ c-d\mathbf i\end{pmatrix}
\right\rangle=-2ac-2bd,\\
X_2
=&\left\langle \mathbf i e_2\psi,\psi\right\rangle
=\left\langle
-\begin{pmatrix}0&-\mathbf i\\ \mathbf i&0\end{pmatrix}
\begin{pmatrix}a+b\mathbf i\\ c+d\mathbf i\end{pmatrix},
\begin{pmatrix}a-b\mathbf i\\ c-d\mathbf i\end{pmatrix}
\right\rangle
=-2ad+2bc,\\
X_3=&\left\langle \mathbf i e_3\psi,\psi\right\rangle
=\left\langle
-\begin{pmatrix}1&0\\ 0&-1\end{pmatrix}
\begin{pmatrix}a+b\mathbf i\\ c+d\mathbf i\end{pmatrix},
\begin{pmatrix}a-b\mathbf i\\ c-d\mathbf i\end{pmatrix}
\right\rangle
=-a^2-b^2+c^2+d^2.
\end{align*}
Next, using the explicit expression
\[
\mathcal J\psi=e_2\overline{\psi}
=
\begin{pmatrix}
c-d\mathbf i\\
-a+b\mathbf i
\end{pmatrix},
\]
we obtain 
\begin{align*}
\langle e_1\mathcal J\psi,\psi\rangle
=&\left\langle
\begin{pmatrix}0&\mathbf i\\ \mathbf i&0\end{pmatrix}
\begin{pmatrix}c-d\mathbf i\\ -a+b\mathbf i\end{pmatrix},
\begin{pmatrix}a-b\mathbf i\\ c-d\mathbf i\end{pmatrix}
\right\rangle
=\bigl(-2ab+2cd\bigr)+\mathbf i\bigl(-a^2+b^2+c^2-d^2\bigr),\\
\langle e_2\mathcal J\psi,\psi\rangle=&
\left\langle
\begin{pmatrix}0&1\\ -1&0\end{pmatrix}
\begin{pmatrix}c-d\mathbf i\\ -a+b\mathbf i\end{pmatrix},
\begin{pmatrix}a-b\mathbf i\\ c-d\mathbf i\end{pmatrix}
\right\rangle=\bigl(-a^2+b^2-c^2+d^2\bigr)+\mathbf i\bigl(2ab+2cd\bigr),\\
\langle e_3\mathcal J\psi,\psi\rangle
=&\left\langle
\begin{pmatrix}\mathbf i&0\\ 0&-\mathbf i\end{pmatrix}
\begin{pmatrix}c-d\mathbf i\\ -a+b\mathbf i\end{pmatrix},
\begin{pmatrix}a-b\mathbf i\\ c-d\mathbf i\end{pmatrix}
\right\rangle
=\bigl(2ad+2bc\bigr)+\mathbf i\bigl(2ac-2bd\bigr).
\end{align*}
This yields
\begin{align*}
A&=(-2ab+2cd,\,-a^2+b^2-c^2+d^2,\,2ad+2bc),\\
B&=(-a^2+b^2+c^2-d^2,\,2ab+2cd,\,2ac-2bd).
\end{align*}

As in Lemma~\ref{Dirac current harmonic}, a direct computation using the Dirac
equation shows that the vector fields $X$, $A$, and $B$ are divergence free.
Moreover, a straightforward calculation shows that they are pairwise orthogonal:
\[
\langle X,A\rangle=\langle X,B\rangle=\langle A,B\rangle=0.
\]
Finally, their lengths agree and satisfy
\[
|X|=|A|=|B|=|\psi|^2=a^2+b^2+c^2+d^2.
\]
This completes the proof.
\end{proof}

Recall from Lemma~\ref{Lemma:main formula} the identity
\begin{align*}
&\Delta_{\Sigma_m}\log|Z_m|
-
\Delta_{\Sigma_{m+1}}\log|Z_{m+1}|\\
=&
\Ric_{\Sigma_m}(\nu_{m+1},\nu_{m+1})
+
|A_{\Sigma_{m+1}}|^2
-
H_{\Sigma_{m+1}}^2 \\
&
+
|Z_m|^{-1}\nabla^{\Sigma_m}_{\nu_{m+1}}
\bigl(\operatorname{div}_{\Sigma_m}Z_m\bigr)
-
\bigl\langle
\nabla_{\Sigma_{m+1}}\log|\nabla_{\Sigma_m}u_m|,
\nabla_{\Sigma_{m+1}}\log|Z_{m+1}|
\bigr\rangle.
\end{align*}

In the special case of a single function $u=u_0$, imposing the normalization
$|Z_1|=1$, Lemma~\ref{Lemma:main formula} reduces to the classical Bochner identity
\[
|\nabla u|\Delta|\nabla u|
=
\Ric(\nabla u,\nabla u)
+
|\nabla^2u|^2
-
|\nabla|\nabla u||^2.
\]
On the other hand, prescribing $|Z_0|=1$ instead yields the stability
inequality for a single minimal hypersurface.

More generally, for a system of $s$ functions one obtains the following result.

\begin{theorem}\label{thm intermediate}
Let $(u_0,\dots,u_{s-1})$ be $C^3$ functions on $M$ satisfying
$du_0\wedge\cdots\wedge du_{s-1}\neq 0$.
Assume that the associated vector fields $Z_m$, defined by
\[
Z_m|Z_{m+1}|^{-1}=\nabla_{\Sigma_m}u_m,
\]
satisfy
\[
\operatorname{div}_{\Sigma_m}Z_m=0
\qquad\text{for all }0\le m\le s-1.
\]
Then
\begin{align*}
&\Delta_{\Sigma}\log|Z_0|
-
\Delta_{\Sigma_s}\log|Z_s|\\
=&
\sum_{m=0}^{s-1}
\bigl(|A_{\Sigma_{m+1}}|^2-H_{\Sigma_{m+1}}^2\bigr)
+
\mathcal C_s(\nu_1,\dots,\nu_s) \\
&
+
\sum_{m=1}^{s-1}
\sum_{p=m+1}^{s}
\sum_{q=p+1}^{n}
\Bigl(
A_{\Sigma_m}(e_p,e_p)A_{\Sigma_m}(e_q,e_q)
-
A_{\Sigma_m}(e_p,e_q)^2
\Bigr) \\
&
+
\sum_{m=0}^{s-1}
\Bigl(
|Z_m|^{-1}\nabla^{\Sigma_m}_{\nu_{m+1}}
(\operatorname{div}_{\Sigma_m}Z_m)
-
\bigl\langle
\nabla_{\Sigma_{m+1}}\log|\nabla_{\Sigma_m}u_m|,
\nabla_{\Sigma_{m+1}}\log|Z_{m+1}|
\bigr\rangle
\Bigr).
\end{align*}
\end{theorem}

Here 
 \[
  \mathcal{C}_m(e_1, \dots, e_m) = \sum_{p=1}^m\sum_{q=p+1}^n 
  \Rm(e_p, e_q, e_p, e_q)
 \] 
 denotes $m$-intermediate curvature, cf. \cite[Definition 1.1]{BrendleHirschJohne24}.

\begin{proof}
Recall from \cite[Lemma~3.8]{BrendleHirschJohne24} the iterated Gauss equation
\begin{align*}
\sum_{m=0}^{s-1}
\Ric_{\Sigma_m}(\nu_{m+1},\nu_{m+1})
&=
\mathcal C_s(\nu_1,\dots,\nu_s) \\
&\quad
+
\sum_{m=1}^{s-1}
\sum_{p=m+1}^{s}
\sum_{q=p+1}^{n}
\Bigl(
A_{\Sigma_m}(e_p,e_p)A_{\Sigma_m}(e_q,e_q)
-
A_{\Sigma_m}(e_p,e_q)^2
\Bigr).
\end{align*}
On the other hand, Lemma~\ref{Lemma:main formula} implies
   \begin{align*}
           &\Delta_{\Sigma}\log|Z_0|-\Delta_{\Sigma_{s}}\log|Z_{s}|\\
        =& \sum_{m=0}^{s-1}\left(\Ric_{\Sigma_m}(\nu_{m+1},\nu_{m+1})+|A_{\Sigma_{m+1}}|^2-H_{\Sigma_{m+1}}^2\right)\\
        &+ \sum_{m=0}^{s-1}\left(|Z_m|^{-1}\nabla^{\Sigma_m}_{\nu_{m+1}}(\operatorname{div}_{\Sigma_m}Z_m)-\langle \nabla_{\Sigma_{m+1}}\log|\nabla_{\Sigma_m}u_m|,\nabla_{\Sigma_{m+1}}\log|Z_{m+1}|\rangle\right).
          \end{align*}
   Combining these identities yields the stated formula.
\end{proof}

For applications, the extrinsic curvature terms and the gradient terms can be
estimated as in \cite{BrendleHirschJohne24}.

\begin{proof}[Proof of Theorem~\ref{thm stern}]
This follows from Theorem~\ref{thm intermediate} with $s=n-2$ together with the
normalization $|Z_s|=1$.
Alternatively, arguing as in the proof of
Theorem~\ref{thm main more general} and using
Lemmas~\ref{Lemma: A-H identity} and~\ref{Lemma:main formula}, one obtains
\begin{align*}
2|Z_0|^{-1}\Delta|Z_0|
-
2\Delta_{\Sigma_{n-2}}\log|Z_{n-2}|
&=
\scal
-
2K(\Sigma_{n-2})
+
\sum_{m=0}^{n-3}
|\nabla_{\Sigma_m}\nu_m|^2
+
|\nabla\log|Z_0||^2,
\end{align*}
where we used that
\[
\sum_{m=0}^{n-3}
\left(
\Ric_{\Sigma_m}(\nu_{m+1},\nu_{m+1})
+
\frac12|A_{\Sigma_{m+1}}|^2
-
\frac12H_{\Sigma_{m+1}}^2
\right)
=
\frac12\scal-\operatorname{K}(\Sigma_{n-2}).
\]
Moreover,
\[
|Z_0|
=
|\nabla u_0|
\,|\nabla_{\Sigma_1}u_1|
\cdots
|\nabla_{\Sigma_{n-3}}u_{n-3}|
=
|du_0\wedge\cdots\wedge du_{n-3}|.
\]
\end{proof}

\begin{remark}\label{rem pde systems}
In case $|Z_{s}|=1$, the PDE system $\operatorname{div}_{\Sigma_m}Z_m=0$, $Z_m|Z_{m+1}|^{-1}=\nabla_{\Sigma_m}u_m$ can be rewritten purely in terms of the functions $u_m$:
\begin{align*}
     \Delta_{\Sigma_m} u_m=  -\sum_{k=m+1}^{s-1}\left\langle\nabla_{\Sigma_m}u_m, \nabla_{\Sigma_m} \log |\nabla_{\Sigma_k}u_k|\right\rangle\qquad \text{for $0\le m\le s-1$}.
\end{align*}
Similarly, for $|Z_0|=1$ we have
\begin{align*}
     \Delta_{\Sigma_m} u_m=\sum_{k=0}^{m}\left\langle\nabla_{\Sigma_m}u_m,\nabla_{\Sigma_m}\log |\nabla_{\Sigma_k}u_k|\right\rangle\qquad \text{for $0\le m\le s-1$}.
\end{align*}
\end{remark}

\begin{example}
As noted in the introduction, Theorem~\ref{thm stern} reduces to Stern's
level-set identity in dimension $n=3$.
The first genuinely new case occurs in dimension $n=4$, where the associated
PDE system takes the simple form
\[
\begin{cases}
\Delta u_0
=
-\langle \nabla u_0,\nabla|\nabla_{\Sigma_1}u_1|\rangle,\\[0.3em]
\Delta_{\Sigma_1}u_1
=
0.
\end{cases}
\]
\end{example}

\bibliographystyle{plain}
\bibliography{literature.bib}

\end{document}